\documentclass[reqno,longbibliography]{amsart}

\usepackage{hyperref}
\hypersetup{
	colorlinks   = true, 
	urlcolor     = blue, 
	linkcolor    = purple, 
	citecolor   = blue 
}

\usepackage{amsmath}
\usepackage{amssymb}
\usepackage{amsthm}
\usepackage{amsfonts}

\usepackage{enumitem}
\usepackage{xcolor}
\usepackage{mathtools}
\usepackage{stackengine}

\usepackage{float}
\usepackage{tikz}
\usetikzlibrary{shapes.misc}
\tikzset{cross/.style={cross out, draw, 
         minimum size=2*(#1-\pgflinewidth), 
         inner sep=0pt, outer sep=0pt}}
\usepackage{IEEEtrantools}
\usepackage{url}

\usepackage[backend=biber, style=trad-plain, doi=false, isbn=false, url=false,eprint=false]{biblatex}
\AtEveryBibitem{\clearlist{language}}
\makeatletter
\@namedef{subjclassname@2020}{%
  \textup{2020} MSC}
\makeatother

\mathcode`l="8000
\begingroup
\makeatletter
\lccode`\~=`\l
\DeclareMathSymbol{\lsb@l}{\mathalpha}{letters}{`l}
\lowercase{\gdef~{\ifnum\the\mathgroup=\m@ne \ell \else \lsb@l \fi}}
\endgroup

\newtheorem{theorem}{Theorem}
\newtheorem{proposition}{Proposition}
\newtheorem{lemma}{Lemma}
\newtheorem{corollary}{Corollary}

\theoremstyle{definition}

\newtheorem*{remark}{Remark}

\newcommand{\details}[1]{}

\newcommand{\Q}{\mathbb{Q}}

\newcommand{\Z}{\mathbb{Z}}
\newcommand{\R}{\mathbb{R}}

\newcommand{\N}{\mathbb{N}}

\newcommand{\bp}{\mathbf{p}}
\newcommand{\bq}{\mathbf{q}}
\newcommand{\bv}{\mathbf{v}}

\newcommand{\bw}{\mathbf{w}}

\newcommand{\bracket}[1]{\langle#1\rangle}
\newcommand{\abs}[1]{\lvert#1\rvert}
\newcommand{\Abs}[1]{\left\lvert#1\right\rvert}
\newcommand{\norm}[1]{\lVert#1\rVert}
\newcommand{\Norm}[1]{\left\lVert#1\right\rVert}
\newcommand{\se}[2]{\{#1,\dots,#2\}}

\newcommand{\eps}{\varepsilon}
\newcommand{\dc}{\dot{c}}

\DeclareMathOperator{\card}{card}

\title{Hausdorff dimension and exact approximation order in $\R^n$}

\author{Prasuna Bandi}
\address{\textbf{Prasuna Bandi} \\
I.H.E.S., Universit\'e Paris-Saclay, Laboratoire Alexandre Grothendieck. 35 Route de Chartres, 91440 Bures-sur-Yvette, France}
\address{Department of Mathematics, University of Michigan, Ann Arbor, MI 48109, USA}
\email{prasuna@umich.edu}

\author{Nicolas de Saxc\'{e}}
\address{\textbf{Nicolas de Saxc\'{e}} \\
CNRS -- LAGA, Universit{\'e} Sorbonne Paris Nord, 99 avenue J.-B. Cl{\'e}ment, 93430 Villetaneuse}
\email{desaxce@math.univ-paris13.fr}

\begin{document}

\date{\today}

\maketitle

\begin{abstract}
Given a non-increasing function $\psi\colon\N\to\R^+$ such that $s^{\frac{n+1}{n}}\psi(s)$ tends to zero as $s$ goes to infinity, we show that the set of points in $\R^n$ that are exactly $\psi$-approximable is non-empty, and we compute its Hausdorff dimension.
For $n\geq 2$, this answers questions of Jarn{\'i}k and of Beresnevich, Dickinson and Velani.
\end{abstract}

\section{Introduction}

Given a non-increasing function $\psi\colon\N\to\R_+^*$, one defines the set of $\psi$-approximable points in $\R^n$ as
\[
W(\psi) = \left\{x\in\R^n\ |\ \mbox{there exists infinitely many}\ \frac{p}{q}\in\Q^n\ \mbox{with}\ \Norm{x-\frac{p}{q}}<\psi(q)\right\},
\]
where the norm on $\R^n$ is given by $\norm{x}=\max_{1\leq i\leq n}\abs{x_i}$ if $x=(x_1,\dots,x_n)$.
\details{This is the problem studied by Jarn{\'i}k, but our method would also apply to the Euclidean norm.}
It follows from Dirichlet's celebrated theorem that for the function
\[
\psi_{\frac{n+1}{n}}(s)=s^{-\frac{n+1}{n}}
\]
one has $W(\psi_{\frac{n+1}{n}})=\R^n$.
On the other hand, for $\tau>0$ and $\psi_\tau(s)=s^{-\tau}$, Jarn{\'i}k~\cite{jarnik} showed the following theorem.

\begin{theorem}[Jarn{\'i}k, 1930]
    Let $n\geq 1$ be an integer.
    For every $\tau\geq\frac{n+1}{n}$, one has 
    \[
    \dim_H W(\psi_\tau)= \frac{n+1}{\tau}.
    \]
\end{theorem}

This shows in particular that if $\tau'>\tau$, then the set $W(\psi_{\tau'})$ is strictly smaller than $W(\psi_{\tau})$.
In fact, for $\tau$ large enough Jarn{\'i}k was able to sharpen this result, as he observed that for every $\tau>2$ and every $c<1$, one even has a strict inclusion
\[
W(c\psi_\tau) \subsetneq W(\psi_\tau).
\]
For $n>1$, the condition $\tau>2$ is unnatural, and Jarn{\'i}k's remarks at the end of his paper suggest that the result should hold for any $\tau>\frac{n+1}{n}$.
One goal of this paper is to show that this is indeed the case.

\bigskip
More precisely, defining the set of \emph{exact $\psi$-approximable} vectors in $\R^n$ by
\[
E(\psi) = W(\psi)\setminus \bigcup_{c<1} W(c\psi),
\]
we shall prove the following result generalizing Jarn{\'i}k~\cite[Satz~6]{jarnik}.

\begin{theorem}[Existence of exact $\psi$-approximable vectors]
\label{thm:existence}
    Let $n\geq 1$ be an integer.
    If $\psi\colon\N\to\R_+^*$ is non-increasing and satisfies $\lim_{s\to\infty}s^{\frac{n+1}{n}}\psi(s)=0$, then $E(\psi)\neq\varnothing$.
\end{theorem}

In the particular case $n=1$, Yann Bugeaud~\cite{bugeaud, bugeaud2} and then Bugeaud and Moreira~\cite{bugeaud3} studied the sets $E(\psi)$ from the point of view of Hausdorff dimension, and showed that $\dim_H E(\psi)=\dim_H W(\psi)$ provided $s^2\psi(s)$ tends to zero at infinity.
Under the assumption that  $\lim_{s\rightarrow \infty}-\frac{\log \psi(x)}{\log(x)}$ exists and is at least $2$, Reynold Fregoli~\cite{fregoli} was able to compute the Hausdorff dimension of $E(\psi)$ in the case $n\geq 3$ but as Jarn{\'i}k himself already observed, the condition $\psi(s)=o(s^{-2})$ is too restrictive when $n\geq 2$, and should be replaced by $\psi(s)=o(s^{-\frac{n+1}{n}})$.
In \cite{BGN} Bandi, Ghosh, and Nandi studied the exact approximation problem in the abstract set-up of Ahlfors regular metric spaces but again, their assumptions imply in particular that the abstract rational points satisfy a certain well-separatedness property, which the rationals in $\R^{n}$ do not satisfy for $n\geq 2$.
A variant of the problem was studied by Beresnevich, Dickinson, and Velani~\cite{bdv} who showed that the set
\[
D(\psi_1,\psi_2) = W(\psi_1)\setminus W(\psi_2)
\]
satisfies $\dim_H D(\psi_1,\psi_2)=\dim_H W(\psi_1)$ under certain assumptions that imply in particular that $\frac{\psi_1(s)}{\psi_2(s)}$ tends to infinity as $s$ goes to infinity.
They observed however that their techniques completely fail if one takes $\psi_2=c\psi_1$, and that new ideas and methods would be needed to cover this case.
Our approach allows us to give a satisfactory answer to this problem, by showing that Bugeaud's result is in fact valid in any dimension. The next theorem is the main result of our paper.

\begin{theorem}[Hausdorff dimension of exact approximable vectors]
\label{thm:main}
    Let $n\geq 1$ be an integer.
    Assume that $\psi\colon\N\to\R_+^*$ is non-increasing and satisfies $\psi(s)=o(s^{-\frac{n+1}{n}})$.
Then the set of exact $\psi$-approximable vectors in $\R^n$ satisfies
    \[
    \dim_H E(\psi) = \dim_H W(\psi).
    \]
\end{theorem}

In \cite{bdv}, the authors also define the set of $\psi$-badly approximable points
\[
\mathbf{Bad}(\psi) = W(\psi)\setminus\bigcap_{c>0} W(c\psi)
\]
and suggest to study the Hausdorff dimension of this set.
We obtain a complete answer to that problem as an immediate corollary of Theorem~\ref{thm:main}.

\begin{corollary}[Hausdorff dimension of $\psi$-badly approximable points]
\label{cor:BA}
Let $\psi\colon\N\to~\R_+^*$ be a non-increasing function such that $\psi(s)=o(s^{-\frac{n+1}{n}})$.
Then $$\dim_H \mathbf{Bad}(\psi) = \dim W(\psi)$$
\end{corollary}

We note however that this corollary can be obtained more easily using the \emph{variational principle in the parametric geometry of numbers} of Das, Fishman, Simmons and Urba{\'n}ski~\cite{dfsu}.
This alternative argument is sketched in paragraphs~\ref{ss:template} and \ref{ss:variational}, as an introduction to the tools and techniques that will be further developed for the proof of Theorem~\ref{thm:main}.
In the particular case of $\psi(s)=s^{-\lambda}$, Corollary~\ref{cor:BA} was obtained independently by Koivusalo, Levesley, Ward and Zhang~\cite{KLWZ} using different methods.

\bigskip
Theorem~\ref{thm:main} above is new for $n\geq 2$ even in the case of the elementary function $\psi(s)=s^{-\lambda}$ for $\lambda>\frac{n+1}{n}$.
In that case, the formula for the Hausdorff dimension is particularly simple: The set $E_\lambda$ of points $x$ in $\R^n$ for which there exist infinitely many rationals $\frac{p}{q}$ such that $\Norm{x-\frac{p}{q}}<q^{-\lambda}$, but only finitely many satisfying $\Norm{x-\frac{p}{q}}<cq^{\lambda}$ if $c<1$, satisfies
\[
\dim_H E_\lambda = \frac{n+1}{\lambda}.
\]
More generally, one defines the \emph{lower order at infinity} of $\psi$, denoted $\lambda_{\psi}$, to be
\begin{equation*}
\lambda_{\psi}:=\liminf_{s\rightarrow \infty} \frac{-\log\psi(s)}{\log s}
\end{equation*}
A result of Dodson~\cite{dodson} shows that if $\lambda_\psi\geq\frac{n+1}{n}$, the dimension of $W(\psi)$ is given by
\[
\dim_H W(\psi) = \frac{n+1}{\lambda_\psi}.
\]
In Theorem~\ref{thm:main}, only the lower bound $\dim_H E(\psi)\geq\dim_H W(\psi)$ requires a proof, and for that we shall construct inside $E(\psi)$ a Cantor set with the required Hausdorff dimension $\frac{n+1}{\lambda_\psi}$.

\smallskip
To construct that Cantor set, the general strategy is similar to the one developed by Bugeaud in~\cite{bugeaud}, using balls of the form $B(y_k,\frac{\psi(H(v_k)}{k})$, where $v_k$ is a rational point and $y_k$ is chosen so that $d(y_k,v_k)=(1-\frac{1}{k})\psi(H(v_k))$.
It is clear that any point $x$ lying in infinitely many such balls is $\psi$-approximable, but not approximated at rate $c\psi$ by the sequence $(v_k)$ if $c<1$.

The difficult point in the proof is to control also the quality of the approximations to $x$ by rational points $v$ that do not appear among the points $v_k$.
Bugeaud's argument for that is based on continued fractions, and uses an elementary separation property for rational points on the real line: If $v_1$ and $v_2$ are two rational numbers with denominator at most $q$, then $d(v_1,v_2)\geq q^{-2}$.
This property is of course also true for rational points in $\R^n$, $n\geq 2$, but one would need the stronger inequality $d(v_1,v_2)\geq q^{-\frac{n+1}{n}}$ if one wanted to use Bugeaud's approach to study $E(\psi)$ for any function $\psi$ such that $\psi(s)=o(s^{-\frac{n+1}{n}})$.
And of course, this stronger separation does not hold for $n\geq 2$.

In order to bypass this problem, the first step is to apply the celebrated \emph{Dani correspondence}, which allows us to translate the exact approximation property of a point $x$ in terms of the behavior of a diagonal orbit of a lattice $\Delta_x$ in $\R^{n+1}$ associated to $x$.
After that, the main ideas we use are borrowed from the \emph{parametric geometry of numbers} developed by Schmidt and Summerer~\cite{ss} and Roy~\cite{roy}, and in particular to the remarkable preprint of Das, Fishman, Simmons and Urba{\'n}ski~\cite{dfsu}, in which the authors explain how to compute the Hausdorff dimension of the set of points whose associated orbits follow a given trajectory in the space of lattices.
We note however that the study of exact approximation requires a precise understanding (see subsection \ref{ss:variational}) of the behavior of an orbit in the space of lattices, whereas the results in \cite{dfsu} only deal with trajectories up to a bounded error term.
For that reason, we need to adapt their methods to our particular problem; what is left is that the branching of our Cantor set is best understood through a certain \emph{template}, encoding the behavior of the diagonal orbits in the space of lattices.

\vspace{.1cm}
\section{Diagonal orbits in the space of lattices}

For the construction of the Cantor set in $E(\psi)$, it will be convenient to interpret the property of exact $\psi$-approximability through the asymptotic behavior of a diagonal orbit in the space of lattices.
This interpretation is given by the Dani correspondence [Theorem 8.5, \cite{KM}].
To make our proof self-contained, we briefly recall and prove the statement that will be used later on.

\bigskip
Note that given any non-increasing function $\psi\colon\N\to\R_+^*$, one may always construct a strictly decreasing function $\psi_1\colon\N\to\R_+^*$ satisfying $(1-\frac{1}{s})\psi(s) \leq \psi_1(s) \leq \psi(s)$, and then interpolate $\psi_1$ to obtain a decreasing function on $\R^+$.
Then, the lower order of $\psi_1$ is $\lambda_{\psi_1}=\lambda_\psi$, and $E(\psi_1)\subset E(\psi)$.
So, in order to prove the desired lower bound on the Hausdorff dimension of $E(\psi)$, it suffices to prove it for $E(\psi_1)$.
This shows that for the proof of Theorem~\ref{thm:main}, we may assume without loss of generality that $\psi$ extends to a decreasing (and continuous) function on $\R^+$.
This assumption makes the statement of Dani's correspondence slightly simpler, so we shall always make it in the sequel.

\subsection{Dani's correspondence}

To any point $x=(x_1,\dots,x_n)$ in $\R^n$ we associate the unipotent matrix
\begin{equation*}
    u_{x}:=
    \begin{pmatrix}
1 \\
-x_{1} & 1 \\
\vdots & & \ddots \\
-x_{n} & & &1
\end{pmatrix}
\end{equation*}
and the unimodular lattice
\[
\Delta_x = u_x\Z^{n+1}
\quad\subset\quad
\R^{n+1}.
\]
The diophantine properties of $x$ are encoded in the asymptotic behavior of the orbit of $\Delta_x$ under the diagonal semigroup $(a_t)_{t>0}$ given by
\begin{equation*}
    a_{t}:=
    \begin{pmatrix}
e^{-t} \\
  & e^{t/n} \\
  & & \ddots \\
  & & & e^{t/n}
    \end{pmatrix}.
\end{equation*}
To state the precise correspondence, we associate to each rational point $v=(\frac{p_{1}}{q},\cdots, \frac{p_{n}}{q})$ in $\Q^{n}$ the primitive integer vector $\bv=(q,p_{1},\cdots, p_{n})$ in $\Z^{n+1}$, where the coordinates $(q,p_1,\dots,p_n)$ are relatively prime.
We shall also use the distance on $\R^n$ given by
\[
d(x,v) = \max_{1\leq i\leq n} \Abs{x_i-\frac{p_i}{q}}
\]
and the height on $\Q^n$ defined by
\[
H(v) = \max\left\{\abs{q},\abs{p_1},\dots,\abs{p_n}\right\}.
\]
In the following, the space $\R^{n+1}$ is endowed with the norm equal to the maximal coordinate in absolute value:
\[
\norm{\bw} = \max_{1\leq i \leq n+1} \abs{\bracket{e_i,\bw}}.
\]
In particular, the equality $\norm{\bw}=\abs{\bracket{e_1,\bw}}$ appearing in item~\ref{back} below means that the largest component of the vector $\bw$ is along the $e_1$ coordinate.

\begin{proposition}[Dani's correspondence]
\label{dani}
Let $\psi:\R^{+}\rightarrow \R^{+}$ be a decreasing function and set $\Psi(s):=\psi(s)^{-\frac{n}{n+1}}$.
Then,
\begin{enumerate}[label=(\alph*)]
    \item If $d(x,v)\leq \psi(H(v))$ for some $v\in\Q^n\cap[0,1]^n$ and $t$ is such that $e^{t}=\Psi(H(v))$, then $\abs{\bracket{e_1,a_tu_x\bv}}=\norm{a_tu_x\bv}$ and
\(
\norm{a_{t}u_{x}\bv}\leq e^{-t}\Psi^{-1}(e^{t}).
\)
    \item \label{back} If $\norm{a_{t}u_{x}\bv}\leq e^{-t}\Psi^{-1}(e^{t})$ and $\abs{\bracket{e_1,a_tu_x\bv}}=\norm{a_tu_x\bv}$ for some $\bv\in\Z^{n+1}$, then the rational point $v$ satisfies $H(v)\leq\Psi^{-1}(e^t)$ and $d(x,v)\leq \psi(H(v))$.
\end{enumerate}
\end{proposition}
\begin{proof} 
Suppose $d(x,v)\leq \psi(H(v))$ and let $t$ be such that
\begin{equation}\label{eqn1}
    e^{t}=\Psi(H(v)).
\end{equation}
Since $v\in[0,1]$, the vector $\bv=(q,p_1,\dots,p_n)$ satisfies $q\geq\max_{1\leq i\leq n}\abs{p_i}$, so $H(v)=q$ and by definition of the height and distance on $\R^{n+1}$,
\[
\norm{a_{t}u_{x}\bv}=\max \{e^{-t}H(v),e^{t/n}H(v)d(x,v)\}
\]
and from the definition of $\Psi$ and our choice of $t$,
\begin{align*}
	e^{t/n}H(v)d(x,v)
	& \leq e^{t/n}H(v)\psi(H(v))\\
	& \leq e^{t/n}H(v)\Psi(H(v))^{-\frac{n+1}{n}}\\
	& \stackrel{\eqref{eqn1}}{=}e^{-t}\Psi^{-1}(e^{t}).
\end{align*}
It follows that $\abs{\bracket{e_1,a_tu_x\bv}}=e^{-t}H(v)=e^{-t}\Psi^{-1}(e^t)=\norm{a_tu_x\bv}$ and
\[
\norm{a_{t}u_{x}\bv}\leq e^{-t}\Psi^{-1}(e^{t}).
\]

For the second item of the proposition, assume that 
\[
\norm{a_{t}u_{x}\bv}=\max \{e^{-t}H(v),e^{t/n}H(v)d(x,v)\}\leq e^{-t}\Psi^{-1}(e^{t}).
\]
Then clearly, $H(v)\leq \Psi^{-1}(e^{t})$, and the condition $\abs{\bracket{e_1,a_tu_x\bv}}=\norm{a_tu_x\bv}$ translates to
\[
e^{t/n}H(v)d(x,v) \leq e^{-t}H(v)
\]
whence
\[
d(x,v) \leq e^{-\frac{n+1}{n}t} \leq \Psi(H(v))^{-\frac{n+1}{n}} = \psi(H(v)).
\]
\end{proof}

\begin{remark}
If one assumes the stronger condition that $\theta\colon H\mapsto H\psi(H)$ is non-increasing, one does not need the condition on $\bracket{e_1,a_tu_x\bv}$ in item~\ref{back}.
Indeed, in that case
\begin{equation*}
    \theta(\Psi^{-1}(e^{t}))= \Psi^{-1}(e^{t}) e^{-\frac{n+1}{n}t} \leq \theta (H(v))= H(v)\psi(H(v)).
\end{equation*}
Hence 
\begin{equation*}
    d(x,v)\leq H(v)^{-1}e^{-\frac{n+1}{n}t}\Psi^{-1}(e^{t})\leq \psi(H(v)).
\end{equation*}
\end{remark}

The above remark in particular allows us to formulate a particularly simple corollary of Dani's correspondence, giving a necessary and sufficient condition on the orbit of the lattice $\Delta_x$ for the point $x$ to belong to the set $E(\psi)$ of exact $\psi$-approximability, under the slightly more restrictive monotonicity condition on $\psi$.

\begin{corollary}
\label{dani-exact}
Given $\psi\colon\R^+\to\R^+$ such that $s\mapsto s\psi(s)$ is non-increasing and $s^{\frac{n+1}{n}}\psi(s)$ tends to zero as $s$ goes to infinity, set $\Psi(s)=\psi(s)^{-\frac{n}{n+1}}$ and
\[
r_{\psi}(t):=-t+\log \Psi^{-1}(e^{t}).
\]
Assume $x$ in $\R^n$ satisfies
 \begin{enumerate} [label=(\alph*)]
     \item \label{cor2a} For every $0<c<1$, for all $t>0$ large enough, $\log \lambda_{1}(a_{t}u_{x}\Z^{n+1})\geq r_{c\psi}(t)$;
     \item \label{cor2b} For arbitrarily large values of $t>0$, one has $\log \lambda_{1}(a_{t}u_{x}\Z^{n+1})\leq r_{\psi}(t)$.
 \end{enumerate}
Then
\(
x \in E(\psi).
\)
\end{corollary}
\begin{proof}
By the first item in Proposition~\ref{dani} the first condition ensures that $x$ does not belong to $W(c\psi)$, for any $c<1$.
The second item together with the above remark shows that $x$ is in $W(\psi)$.
\end{proof}

\subsection{A template for \texorpdfstring{$\lambda_1$}{}}
\label{ss:template}

In order to simplify the presentation of this introductory paragraph, we shall assume that the map $s\mapsto s\psi(s)$ is decreasing.
Recall from the preceding section that $\Psi(s)=\psi(s)^{-\frac{n}{n+1}}$ and
\[
r_\psi(t)=-t+\log\Psi^{-1}(e^t).
\]
Together with the condition $\psi(s)=o(s^{-\frac{n+1}{n}})$, our monotonicity assumption on $s\psi(s)$ ensures that
\begin{enumerate}
\item $t\mapsto r_\psi(t)-\frac{t}{n}$ is decreasing\details{ because $r_\psi(t)-\frac{t}{n}=\log(e^{-t\frac{n+1}{n}}\Psi^{-1}(e^t))=\log(u\psi(u))$, where $u=\Psi^{-1}(e^t)$};
\item $t\mapsto r_\psi(t)+t$ is increasing\details{ because $r_\psi(t)+t=\log\Psi^{-1}(e^t)$};
\item $\lim_{t\to+\infty} r_\psi(t)=-\infty$\details{ because $\psi(s)=o(s^{-\frac{n+1}{n}})$}.
\end{enumerate}
Given a point $x$ in $\R^{n}$, we define the function
\[
\begin{array}{cccc}
c_x\colon & \R^+ & \to & \R\\
& t & \mapsto & c_x(t)=\log\lambda_1(a_tu_x\Z^{n+1}).
\end{array}
\]
It follows from Corollary~\ref{dani-exact} that in order to prove Theorem~\ref{thm:existence} under the additional assumption that $s\mapsto s\psi(s)$ is decreasing, it suffices to construct a point $x$ in $\R^n$ for which the function $c_x$ satisfies the two conditions
\begin{equation}\label{c_condition}
\left\{
\begin{array}{ll}
\forall c<1,\ \forall t>0\ \mbox{sufficiently large},\ c_x(t) \geq r_{c\psi}(t)\\
\exists t>0\ \mbox{arbitrarily large}:\ c_x(t)=r_\psi(t).
\end{array}
\right.
\end{equation}

The \emph{parametric geometry of numbers}, introduced by Schmidt and Summerer~\cite{ss}, gives a combinatorial description of the function $c_x$ on $\R^+$.
It implies in particular that there exists a continuous affine by parts function $T_x\colon\R^+\to\R^-$, with slopes in $\{-1,0,\frac{1}{n}\}$ such that the difference $c_x-T_x$ remains bounded on $\R^+$.
Conversely, one may start from such a \emph{template} $T$ and try to construct a point $x$ in $\R^n$ such that $c_x$ stays at bounded distance from $T$; Schmidt and Summerer gave necessary combinatorial conditions on $T$ for the existence of such a point $x$, and Roy~\cite{roy} showed that these conditions are also sufficient.
Finally, Das, Fishman, Simmons and Urba{\'n}ski~\cite{dfsu} gave a formula for the Hausdorff dimension of the set of points $x$ in $\R^n$ following a given template $T$.
Our proof of Theorems~\ref{thm:existence} and \ref{thm:main} is much inspired by this parametric geometry of numbers: We shall give ourselves a template $T$ satisfying conditions~\eqref{c_condition} above and then construct points that follow closely this model trajectory.
The general picture can be seen in Figure~\ref{T_and_rpsi} below.

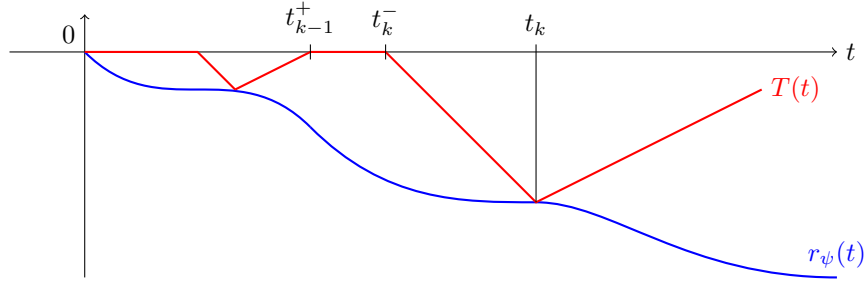
\begin{figure}[h!]
\begin{center}
\begin{tikzpicture}
\draw (0,0) node[above left] {0};
\draw[->] (-1., 0) -- (10, 0) node[right] {$t$};
\draw[<-] (0, .5) -- (0, -3);

\draw[color=blue, thick] (0,0) .. controls (1,-1) and (2,0) .. (3,-1);
\draw[color=blue, thick] (3,-1) .. controls (4,-2) and (5,-2) .. (6,-2);
\draw[color=blue, thick] (6,-2) .. controls (7,-2) and (8,-3) .. (10,-3);
\draw[color=blue] (10,-3) node[above] {$r_\psi(t)$};
\draw[color=red, thick] (0,0) -- (1.5,0) -- (2,-.5);
\draw[color=red, thick] (2,-.5) -- (3,0) -- (4,0) -- (6,-2) -- (9,-.5);
\draw[color=red] (9,-.5) node[right] {$T(t)$};

\draw (3,-.1) -- (3,.1) node[above] {$t_{k-1}^+$};
\draw (4,-.1) -- (4,.1) node[above] {$t_{k}^-$};
\draw (6,-2) -- (6,.1) node[above] {$t_{k}$};

\end{tikzpicture}
\end{center}
\caption{Template $T$ above the graph of $r_\psi$.}
\label{T_and_rpsi}
\end{figure}

Observe also that the lower order at infinity $\lambda_\psi$ of the function $\psi$ can be read-off $r_\psi$ through the formula
\[
\gamma_\psi := \liminf\frac{-r_{\psi}(t)}{t} = \frac{n\lambda_\psi-n-1}{n\lambda_\psi}.
\]
\details{
Indeed, $\gamma_\psi= 1-\limsup \frac{\log\Psi^{-1}(e^t)}{t} = 1-\limsup\frac{\log\Psi^{-1}(q)}{\log q}= 1-\limsup\frac{\log s}{\log\Psi(s)}= 1-\frac{n+1}{n}\limsup\frac{\log s}{\log1/\psi(s)}$.}
Moreover, if $q_k$ is an increasing sequence of denominators such that $\lambda_\psi=\lim \frac{\log1/\psi(q_k)}{\log q_k}$, then, setting $t_k=\log\Psi(q_k)$, one has
\[
\gamma_\psi = \lim \frac{-r_\psi(t_k)}{t_k}.
\]
We now fix such an increasing sequence $(t_k)_{k\geq 1}$ and, taking a subsequence if necessary, we shall also assume that it tends to infinity sufficiently fast.
We also let
\[
t_k^- = t_k + r_\psi(t_k)
\quad\mbox{and}\quad
t_k^+ = t_k - n r_\psi(t_k).
\]
Provided $(t_k)$ increases fast enough, one always has
\[
0 < t_1^- < t_1 < t_1^+ < t_{2}^- < \ldots
\]
and we define a function $T\colon\R^+\to\R^-$ with slopes in $\{-1,0,\frac{1}{n}\}$ by $T(0)=0$ and
\[
\frac{dT}{dt}(t) = \left\{
\begin{array}{lll}
0 & \mbox{if} & t_{k-1}^+ < t < t_k^-\\
-1 & \mbox{if} & t_k^- < t < t_k\\
\frac{1}{n} & \mbox{if} & t_k < t < t_k^+.
\end{array}
\right.
\] 
Note that this function is continuous and satisfies $T(t_k)= r_\psi(t_k)$ for each $k\geq~1$.
\details{Indeed, starting from $T(t_k^-)=0$, one gets $T(t_k)=T(t_k^-) -(t_k-t_k^-) = r_\psi(t_k)$ and then $T(t_k^+)=T(t_k)+\frac{1}{n}(t_k^+-t_k)=0$.}
Moreover, it follows from the properties of $r_\psi$ listed at the beginning of this paragraph that $T(t)\geq r_\psi(t)$ for all $t>0$.
\details{By construction $T(t_k)=r_\psi(t_k)$. Then, if $t_k^-<t<t_k$, write $T(t)=r_\psi(t_k)+(t_k-t)\geq r_\psi(t)$ because $t\mapsto r_\psi(t)+t$ is increasing.
And if $t_k<t<t_k^+$, write $T(t)=r_\psi(t_k)+\frac{1}{n}(t-t_k)\geq r_\psi(t)$ because $t\mapsto r_\psi(t)-\frac{t}{n}$ is decreasing.}

\subsection{The variational principle and beyond}
\label{ss:variational}

By Proposition~\ref{dani}, if $x$ is a point in $\R^n$ such that $c_x$ remains at bounded distance from the template $T$ constructed above, then there exist constants $C$ and $c>0$ such that $x$ lies in $W(C\psi)$ but not in $W(c\psi)$.
By a result of Damien Roy~\cite[Theorem~1.3]{roy}, there exists a point $x$ in $\R^n$ such that $c_x(t)=T(t)+O(1)$ and this shows that the set $W(C\psi)\setminus W(c\psi)$ is non-empty if $C$ is large and $c>0$ small enough.
Replacing the template $T$ by $T-R_0$, where $R_0$ is some large positive constant, this argument can be modified slightly to show that if $c>0$ is small enough, then
\[
W(\psi)\setminus W(c\psi) \neq \varnothing.
\]
In fact, the variational principle of Das, Fishman, Simmons and Urba\'nski~\cite[Theorem~2.3]{dfsu} can be used to compute the Hausdorff dimension of the set $D_T$ of points $x$ in $\R^n$ such that $c_x$ follows the template $T$ up to some bounded error: If the sequence $(t_k)$ tends to infinity fast enough, one finds
\[
 \dim_H \, D_T= n \left(1 - \lim_{k\to+\infty} \frac{- r_\psi(t_k)}{t_k}\right)=  \frac{n+1}{\lambda_\psi} = \dim_H W(\psi).
\]
Again, the Dani correspondence allows one to translate this into the following slightly weaker version of Corollary~\ref{cor:BA} from the introduction.
Since this theorem can also be seen as an immediate consequence of Theorem~\ref{thm:main}, we do not include full details for the proof sketched above.

\begin{corollary}[Hausdorff dimension of $\psi$-badly approximable points]
Assume $\psi\colon\R^+\to\R^+$ is such that $s\mapsto s\psi(s)$ is decreasing and $\psi(s)=o(s^{-\frac{n+1}{n}})$.
Then, the set of badly approximable numbers in $\R^n$, defined as
\[
\mathbf{Bad}(\psi) = W(\psi)\setminus\bigcap_{c>0} W(c\psi)
\]
satisfies $\dim_H \mathbf{Bad}(\psi) = \dim W(\psi)$.
\end{corollary}

Unfortunately, the available results from parametric geometry of numbers, such as the above-cited \cite{roy} or \cite{dfsu}, only give information on the behavior of $c_x$ up to some bounded additive constant.
In contrast, to construct a point in $E(\psi)$, the Dani correspondence shows that one needs to understand $c_x(t)$ within an error term that goes to zero as $t$ tends to infinity, at least at the times $t$ where $c_x(t)$ approaches $r_\psi(t)$.
In order to do so, we shall use a Cantor set construction, similar in spirit to the one used \cite{dfsu}, but with better control on $c_x(t)$ when it takes large negative values.

\section{A Cantor set in \texorpdfstring{$E(\psi)$}{}}

Our goal is now to prove Theorem~\ref{thm:main}.
Throughout this section, $\psi\colon\R^+\to\R^+$ denotes a decreasing function with lower order at infinity equal to $\lambda_\psi$, and we assume without loss of generality that $\lambda_\psi<+\infty$.
We shall construct a Cantor set $E_\infty$ inside $E(\psi)$ with Hausdorff dimension $\tfrac{n+1}{\lambda_\psi}$.
The definition of the level sets of $E_\infty$ is based on the behavior of the maps $c_x$.
But before turning to the actual construction, we explain what important properties of that map we need in order to ensure that $E_\infty$ is indeed included in $E(\psi)$.

\subsection{Main properties of \texorpdfstring{$E_\infty$}{}}

Just as in the previous section, the sequence of times $(t_k)$ is assumed to satisfy
\[
\lim_{k\to\infty}\frac{r_\psi(t_k)}{t_k} = \limsup_{t\to\infty}\frac{r_\psi(t)}{t} := -\gamma_\psi
\]
where $\Psi(s)=\psi(s)^{-\frac{n}{n+1}}$ and 
\(
r_\psi(t) = -t + \log \Psi^{-1}(e^t).
\)\\
Let
\[
M_k = -\sup_{t\geq t_{k-1}} r_\psi(t).
\]
Note that this definition implies that $M_k\leq -r_\psi(t_{k-1})$.
We shall assume that $t_k$ is sufficiently large compared to $t_{k-1}$ in order to ensure that $M_k$ is very small compared to $-r_\psi(t_k)$; this parameter $M_k$ will then be used to define small intervals around $t_k$ or $r_\psi(t_k)$.

In the sequel, we use three constants:
\begin{itemize}
\item $R_0\geq 1$ depending only on $n$;
\item $R_1$ depending on $\gamma_\psi$ and $R_0$;
\item $R_2$ depending on $n$, $R_0$ and $R_1$.
\end{itemize}
Then we define $t_k^-<t_k$ and $t_k^+>t_k$ by
\[
t_k^- = t_k + r_\psi(t_k)
\quad\mbox{and}\quad
t_k^+ = t_k + R_2 M_k.
\]
We shall construct a Cantor set $E_\infty$ of points $x$ for which the trajectory $c_x$ has the following two properties:
\vspace{0.4cm}
\begin{enumerate}[label=(\Alph*)]
\item \label{a}
    For all $t\in[t_{k-1}^+, t_k^--4R_0M_k]$, $c_x(t) \geq -M_{k}$;
\item \label{b}
    For each $k$, there exists $\bv_k=\bv_k(x)\in\Z^{n+1}$ such that
    \[
   \quad \forall t\in [t_{k}^-,t_{k}^+],\quad   c_x(t) = \log \norm{a_tu_x \bv_k}
    \]
    and moreover, the point $v_k$ in $\Q^n$ corresponding to $\bv_k$ satisfies
    \[
    \left\{\begin{array}{ll}
	(i) &    e^{t_k^--5R_0M_k} \leq H(v_k) \leq e^{t_k^--3R_0M_k}\\
	(ii) &    \left(1-\frac{1}{k}\right)\psi(H(v_k)) < d(v_k,x) < \psi(H(v_k))\\
	(iii) &    t_k-R_1M_k < t_k^x:=-\frac{n}{n+1}\log d(v_k,x) < t_k+1.
    \end{array}\right.
    \]
\end{enumerate}

\begin{figure}[h!]
\label{cxtkminustkplus}
\begin{center}
\begin{tikzpicture}
\draw[->] (-1., 0) -- (10, 0) node[right] {$t$};
\draw[<-] (0, .5) -- (0, -3);

\draw[color=blue, thick] (0,0) .. controls (1,-.7) and (2,0) .. (3,-.5);
\draw[color=blue, thick] (3,-.5) .. controls (4,-1) and (5,-2.5) .. (6,-2);
\draw[color=blue, thick] (6,-2) .. controls (7,-1.5) and (8,-3) .. (10,-3);
\draw[color=blue] (10,-3) node[above] {$r_\psi(t)$};

\draw (6,-2) -- (6,.1) node[above] {\tiny{$t_k$}};
\draw (6,-2) --  (4,0) ;
\draw (4,-.1) -- (4,.1) node[above] {\tiny{$t_k^-$}};
\draw (3.2,-.1) -- (3.2,.1) node[above] {\tiny{$t_k^-\!\!\!-\!\!5R_0M_k$}};
\draw (1,-.3) -- (1,.1) node[above] {\tiny{$t_{k-1}^+$}};
\draw (6.7,-1.9) -- (6.7,.1) node[above] {\tiny{$t_k^+$}};
\draw (5.7,-2.1) -- (5.7,.1) node[above] {\tiny{$t_k^x$}};

\draw (10,-.3) -- (-.1,-.3) node[left] {$-M_{k}$};
\draw (10,-1.9) -- (-.1,-1.9) node[left] {$-M_{k+1}$};

\draw[color=red, thick] (1,-.25) -- (1.5,0) -- (1.8,-.3) -- (2.4,0) -- (2.7,-.3) -- (3.1,-.1) -- (3.3,-.3) -- (3.7,-.1);
\draw[color=red,thick] (3.7,-.1) -- (5.7,-2.1) -- (6.7,-1.6);
\draw[color=red] (6.7,-1.6) node[right] {$c_x(t)$};

\end{tikzpicture}
\end{center}
\caption{Graph of $c_x$ on the interval $[t_k^-,t_k^+]$.}
\end{figure}
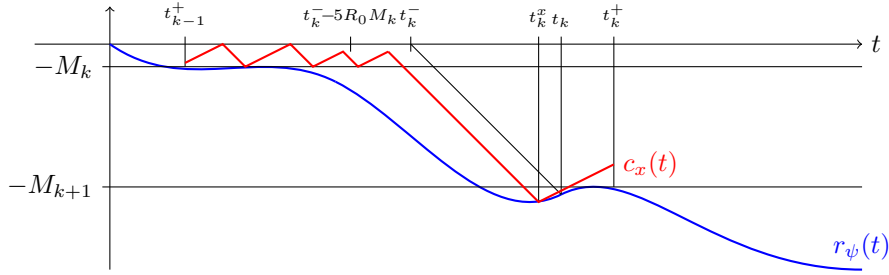

Let us first use Dani's correspondence to check that the two conditions above ensure that $E_\infty$ will be a subset of $E(\psi)$.

\begin{lemma}\label{tkminusr}
Let $x$ in $\R^n$ be such that $c_x$ satisfies \ref{a} and \ref{b} for all $k\geq 1$.
Then $x\in E(\psi)$.
\end{lemma}
\begin{proof}
By construction, $d(v_k,x)<\psi(H(v_k))$ for all $k$, so $x$ lies in $W(\psi)$.
Let us show that if $c<1$, then $x$ does not belong to $W(c\psi)$.
We use the precise form of Dani's correspondence given by Proposition~\ref{dani} and show that for all $t>0$ large enough, if $\bv\in\Z^{n+1}$ satisfies $\abs{\bracket{e_1,a_tu_x\bv}}=\norm{a_tu_x\bv}$, then
\[
\norm{a_tu_x\bv} \geq e^{-t}\Psi_c^{-1}(e^t),
\]
where $\psi_c(s)=c\psi(s)$ and $\Psi_c(s)=\psi_c(s)^{-\frac{n}{n+1}}$.
\details{This requirement is weaker than $c_x(t)\geq r_{c_k\psi}(t) = -t + \log\Psi_{c_k}^{-1}(e^t)$ for all $t\in[t_k^-,t_k^+]$, because we only consider vectors for which $\abs{\bracket{e_1,a_tu_x\bv}}=\norm{a_tu_x\bv}$.
Essentially, we need only worry about times where $\dc_x(t)<0$.
The point is that we shall use the precise form of Dani's correspondence given by Proposition~\ref{dani} rather than the simpler Corollary~\ref{dani-exact}.
}

If $t$ belongs to some interval $[t_{k-1}^+,t_k^--4R_0M_k]$, this follows from \ref{a} and the definition of $M_k$, since
\[
c_x(t) \geq -M_k \geq r_\psi(t).
\]
We can then easily extend this property to a lower bound on the interval $[t_k^- ~-~4R_0M_k,t_k^-]$ for which one has
\[
c_x(t) \geq c_x(t_k^--4R_0M_k)-4R_0M_k \geq -5R_0M_k \geq r_\psi(t)
\]
provided the sequence $(t_k)$ was chosen to increase sufficiently fast in order to ensure that $\sup_{t\geq t_k^--4R_0M_k}r_\psi(t)<-5R_0M_k$.
\details{This is feasible since $r_\psi(t)$ tends to $-\infty$ and $t_k^-$ goes to $+\infty$ as $t_k$ goes to $+\infty$, because we assumed $\gamma_\psi<1$.}\\
Now assume $t\in [t_k^-,t_k^+]$.
Note that
\[
c_x(t)  = -t + \log H(v_k) + \max\left(0,\tfrac{n+1}{n}t+\log d(v_k,x)\right).
\]
Therefore, the condition $\abs{\bracket{e_1,a_tu_x\bv_k}}=\norm{a_tu_x\bv_k}$ is met only if $t\leq-\frac{n}{n+1}\log d(v_k,x)$.
For any such $t$, one has
\[
c_x(t) = -t + \log H(v_k).
\]
From $d(v_k,x)\geq (1-\frac{1}{k})\psi(H(v_k))$, we infer that $t\leq -\frac{n}{n+1}\log\left( (1-\frac{1}{k})\psi(H(v_k))\right)$ i.e.
\[
H(v_k) \geq \Psi_{(1-\frac{1}{k})}^{-1}(e^t).
\]
Thus, we find
\[
c_x(t) \geq -t + \Psi_{(1-\frac{1}{k})}^{-1}(e^t) = r_{(1-\frac{1}{k})\psi}(t).
\]
If $k$ is large enough so that $1-\frac{1}{k}\geq c$, this shows that the condition $\abs{\bracket{e_1,a_tu_x\bv_k}}=\norm{a_tu_x\bv_k}$ implies $\norm{a_tu_x\bv_k}\geq e^{-t}\Psi_c^{-1}(e^t)$.
Moreover, for any integer vector $\bv$ linearly independent with $\bv_k$, one can use Minkowski's second theorem to bound, for $t_k^-\leq t\leq t_k^x:=-\frac{n}{n+1}\log d(v_k,x)$,
\begin{align*}
\log\norm{a_tu_x\bv} & \geq \log\norm{a_{t_k^-}u_x\bv} +\frac{1}{n}(t-t_k^-) - O_n(1)\\
& \geq -5R_0M_k - O_n(1)\\
& \geq r_\psi(t)\geq r_{c\psi}(t).
\end{align*}
Then, we note that item (iii) of condition \ref{b} and the definition of $t_k^+$ imply $t_k^+-t_k^x\leq (R_1+R_2)M_k$ so that for $t_k^x \leq t\leq t_k^+$,  
\begin{align*}
\log\norm{a_tu_x\bv} & \geq \log\norm{a_{t_k^x}u_x\bv} - (t-t_k^x)\\
& \geq -5R_0M_k - O_n(1) - (R_1+R_2)M_k \\
& \geq r_\psi(t)\geq r_{c\psi}(t)
\end{align*}
assuming again that the $(t_k)$ increase sufficiently fast to ensure that $-(5R_0+R_1+R_2)M_k-O_n(1)\geq r_\psi(t)$ for all $t\geq t_k-R_1M_k$.

This proves the desired inequality, and by the first part of Proposition~\ref{dani}, we obtain that for every $c<1$, any rational point $v$ sufficiently close to $x$ satisfies
\[
d(v,x) \geq c\psi(H(v)).
\]
\end{proof}

\details{
\begin{remark}
Assuming $s\psi(s)$ is decreasing, one can get a lower bound $c_x(t)\geq r_{c\psi}(t)$ for all large enough $t>0$.
Indeed, the function $r_{\psi}(t)-\frac{t}{n}$ is then decreasing, so for $t_k^+\geq t\geq t_k^x =-\frac{n}{n+1}\log d(v_k,x)$, one has
\[
c_x(t) = \frac{t}{n} + \log H(v_k) + \log d(v_k,x)
\]
whence
\[
c_x(t) - r_{c_k\psi}(t) \geq c_x(t_k^x)-r_{c_k\psi}(t_k^x) \geq 0.
\]
But without this stronger monotonicity condition, this inequality may fail.
(Draw a picture of a function $r_\psi(t)$ for which it fails.)
\end{remark}
}
\details{
\begin{figure}[h!]
\label{cxtkminustkplus}
\begin{center}
\begin{tikzpicture}
\draw[->] (-1., 0) -- (10, 0) node[right] {$t$};
\draw[<-] (0, .5) -- (0, -3);
\draw[color=blue, thick] (0,0) .. controls (3,-3) and (3,-1) .. (3,-.2);
\draw[color=blue, thick] (3,-.2) .. controls (7,-4.2) and (7,-3) .. (7,-.5);
\draw[color=blue, thick] (7,-.5) -- (10,-3.5);
\draw[color=blue] (10,-3) node[above] {$r_\psi(t)$};
\draw[color=red, thick] (0,0) -- (2.5,0) -- (2.98,-.5) -- (4,0) -- (6,0) -- (7,-1) -- (9,0) -- (10,0) node[below] {$c_x(t)$};
\draw (3,-.1) -- (3,.1) node[above] {\tiny{$t_{k-1}$}};
\draw (7,-.1) -- (7,.1) node[above] {\tiny{$t_k$}};
\end{tikzpicture}
\end{center}
\caption{There exist times $t$ for which $c_x(t)$ is much lower than $r_\psi(t)$ (and also $r_{c\psi}(t)$), but this never happens when $\dot{c}_x(t)<0$.}
\end{figure}
}

\subsection{Construction of the Cantor set}
\label{ss:ccs}

Having fixed some large $M>0$ such that $N=e^{\frac{n+1}{n}M}$ is an integer, the Cantor set $E_\infty$ will be obtained as a decreasing intersection 
\begin{equation*}
    E_{\infty}:=\bigcap_{l=0}^{\infty}E_{l}
\end{equation*}
where $E_l$, the $l$-th \textit{level} of the Cantor set, is a finite union of disjoint cubes of sidelength~$N^{-l}$.
Each set $E_l$ is defined inductively so that for all $x$ in $E_l$, the above properties~\ref{a} and \ref{b} are satisfied up to time $t=lM$.
More precisely, we shall check that for arbitrarily large $l$, for every $x$ in $E_{l}$ and every $k\geq 1$, we have:
\begin{enumerate}[label=(\Alph*$_l$)]
\item \label{al}
    For all $t\in\se{0}{lM}\cap [t_{k-1}^+, t_k^--4R_0M_k]$, then $c_x(t) \geq -M_{k}+M$;
\item \label{bl}
    For each $k$, there exists $\bv_k=\bv_k(x)\in\Z^{n+1}$ such that
    \[
   \quad \forall t\in [0,lM]\cap [t_{k}^-,t_{k}^+],\quad   c_x(t) = \log \norm{a_tu_x \bv_k}
    \]
    and moreover, the point $v_k$ in $\Q^n$ corresponding to $\bv_k$ satisfies
    \[
    \left\{\begin{array}{ll}
	(i) &    e^{t_k^--5R_0M_k} \leq H(v_k) \leq e^{t_k^--3R_0M_k}\\
	(ii) &    \left(1-\frac{1}{k}\right)\psi(H(v_k)) < d(v_k,x) < \psi(H(v_k))\\
	(iii) &    t_k-R_1M_k < t_k^x=-\frac{n}{n+1}\log d(v_k,x) < t_k + 1.
    \end{array}\right.
    \]
\end{enumerate}

\begin{remark}
In \ref{al}, it is enough to consider times in $\se{0}{lM}$, because for $kM\leq t <(k+1)M$, one has $c_{x}(t)\geq c_{x}(kM)-M$.
\end{remark}

Let us first show that condition (iii) follows from (i) and (ii) if $R_1$ and the sequence $(t_k)$ are chosen appropriately.
Recall that we assumed $\lambda_\psi<+\infty$, which is equivalent to
\[
\gamma :=\liminf-\frac{1}{t}r_\psi(t)<1.
\]

\begin{lemma}\label{c1}
Let
$
R_1 = \frac{10R_0}{1-\gamma}
$
and assume that $t_k$ is chosen so that 
\begin{equation}\label{tkminus}
 r_\psi(t_k-R_1M_k) \leq r_\psi(t_k) + \frac{1+\gamma}{2}R_1M_k.
\end{equation}
Then condition (iii) from (B$_l$) above is implied by (i) and (ii).
\end{lemma}

\begin{figure}[h!]
\begin{center}
\begin{tikzpicture}
\draw[->] (-1., 0) -- (10, 0) node[right] {$t$};
\draw[<-] (0, .5) -- (0, -3);

\draw (9,-3) -- (9,.1) node[above] {\tiny{$t_k$}};
\draw (6,-.1) -- (6,.1) node[above] {\tiny{$t_k^-$}};
\draw (5,-.1) -- (5,.1) node[above] {\tiny{$t_k^-\!\!\!-\!\!5R_0M_k$}};

\draw[color=blue] (9,-3) --  (6,0) ;
\draw[color=blue] (5,0) -- (8,-3) node[left] {\tiny{slope $-1$}};
\draw[color=red] (7,-2) -- (9,-3) node[right] {\tiny{slope $-\frac{1+\gamma}{2}$}};
\draw (7,-2) -- (7,.1) node[above] {\tiny{$t_k-R_1M_k$}};
\draw (7,-2) node[left] {$A$};
\end{tikzpicture}
\end{center}
\caption{Controlling $t_k^x$.}
\label{choicer1}
\end{figure}
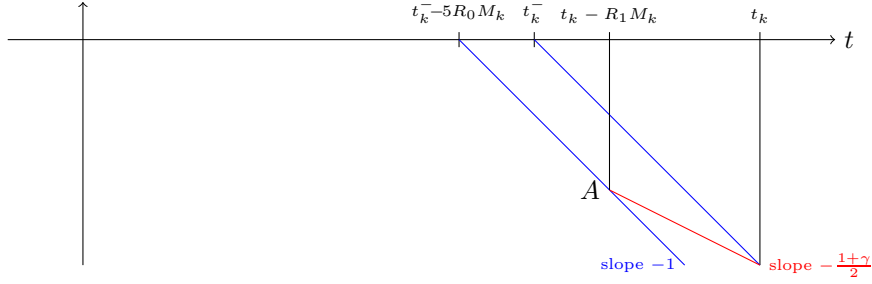

\begin{proof}
First, observe that since $d(v_k,x)\geq (1-\tfrac{1}{k})\psi(H(v_k))$ and $H(v_k)\leq e^{t_k^-}$, one has
\begin{align*}
t_k^x & = -\frac{n}{n+1}\log d(v_k,x)\\
	& \leq \log \Psi(H(v_k)) + \frac{n}{n+1}\log\frac{k}{k-1}\\
	& \leq \log \Psi(e^{t_k^-}) + 1 = t_k + 1.
\end{align*}
For the other inequality, we use Figure~\ref{choicer1} above.
By \eqref{tkminus} the graph of $r_\psi$ passes below the point $A = (t_k-R_1M_k,r_\psi(t_k)+\frac{1+\gamma}{2}R_1M_k)$ and therefore remains below the line of slope $-1$ passing through $A$ on the interval $[0,t_k-R_1M_k]$.
On the other hand, on the interval $[t_k^-,t_k^x]$, one has
\[
\log\norm{a_tu_x\bv_k} = - t + \log H(v_k)
\]
so the graph of $t\mapsto\log\norm{a_tu_x\bv_k}$ follows a line of slope $-1$ that intersects the $t$-axis between $t_k^--5R_0M_k$ and $t_k^-$, by \ref{bl}(i).
Now $R_1=\frac{10R_0}{1-\gamma}$ was chosen so that the line of slope $-1$ intersecting the $t$-axis at $t_k^--5R_0M_k$ passes through $A$, so we may conclude that the graphs of $r_\psi$ and $t\mapsto\log\norm{a_tu_x\bv_k}$ cannot meet before time $t_k-R_1M_k$.
But from $d(x,v_k)<\psi(H(v_k))$ we know that $\log\norm{a_{t_k^x}u_x\bv_k} = -t_k^x + \log H(v_k) < r_\psi(t_k^x)$ and thus $t_k^x\geq t_k-R_1M_k$.
\details{
Indeed,
$\Psi(H(v_k))\leq d(x,v_k)^{-\frac{n}{n+1}}$
so
$\log\Psi(H(v_k)) \leq t_k^x$
and
$-t_k^x+\log H(v_k) \leq -t_k^x + \log \Psi^{-1}(e^{t_k^x}) = r_\psi(t_k^x)$.
}
\end{proof}

We now justify that it is indeed possible to choose the times $t_k$ inductively so that \eqref{tkminus} is always satisfied.
In fact, for the proof of Theorem~\ref{thm:main}, in addition to the condition $\lim\tfrac{-r_\psi(t_k)}{t_k}=\gamma_\psi$, we shall need to control $r_\psi(t)$ on all times after $t_k$, and on times shortly before $t_k$ as well.
For this the times $t_k$, $k\geq 1$ are chosen inductively so that, with
\(
-M_k = \sup_{t\geq t_{k-1}} r_\psi(t),
\)
one has,
\begin{enumerate}[label=(\Roman*)]
\item\label{limsup}
\(
-\gamma_\psi - \frac{1}{k} \leq \frac{r_\psi(t_k)}{t_k} \leq -\gamma_\psi + \frac{1}{k};
\)
\item\label{aftertk}
\(
\forall t\geq t_k,\quad r_\psi(t) \leq r_\psi(t_k) + R_1M_k;
\)
\item 
\(
 r_\psi(t_k-R_1M_k) \leq r_\psi(t_k) + \frac{1+\gamma}{2}R_1M_k.
\)
\end{enumerate}

The lemma below ensures that it is indeed feasible to choose $(t_k)$ in the desired way.

\begin{lemma}\label{finalchoicetk}
Assume $\gamma_\psi=\gamma<1$ and $t_{k-1}$ has been defined.
Given $R>0$ (possibly depending on $t_{k-1}$), we may always choose $t_k$ arbitrarily large so that
\begin{equation}\label{limsup2}
-\gamma_\psi - \frac{1}{k} \leq \frac{r_\psi(t_k)}{t_k} \leq -\gamma_\psi + \frac{1}{k}
\end{equation}
and
\begin{equation}\label{minusr2}
r_\psi(t_k-R) \leq r_\psi(t_k)+\frac{1+\gamma}{2}R
\end{equation}
and for all $t\geq t_k$,
\begin{equation}\label{aftertk2}
r_\psi(t) \leq r_\psi(t_k) + R.
\end{equation}
\end{lemma}
\begin{proof}
Let $0<\eps_k<\min(\tfrac{1}{k},\frac{1}{8})$, and start with $t_k^{(0)}$ such that
\[
\forall t\geq\frac{t_k^{(0)}}{3},\quad \frac{r_\psi(t)}{t} < -\gamma+\eps_k
\quad\mbox{and}\quad
\frac{r_\psi(t_k^{(0)})}{t_k^{(0)}} \geq -\gamma - \eps_k.
\]
Replacing $t_k^{(0)}$ by the largest time $t$ for which $r_\psi(t)=r_\psi(t_k^{(0)})$ if necessary, we may also assume that
\[
\forall t\geq t_k^{(0)},\quad r_\psi(t) \leq r_\psi(t_k^{(0)}).
\]

For $i\geq 1$, we define inductively $t_k^{(i)}$ in the following way.
Assuming, $t_k^{(i-1)}$ has been defined, if $r_\psi(t_k^{(i-1)}-R) \leq r_\psi(t_k^{(i-1)}) + \frac{1+\gamma}{2}R$, then we stop and let $t_k=t_k^{(i-1)}$.
Otherwise, let
\[
t_k^{(i)} = t_k^{(i-1)} -R.
\]
This procedure must stop for some $i\leq \frac{4\eps_kt_k^{(0)}}{R}$, otherwise we would have, for $i=\lceil \frac{4\eps_kt_k^{(0)}}{R}\rceil$,
\begin{align*}
r_\psi(t_k^{(i)}) & \geq \tfrac{1+\gamma}{2}iR - (\gamma+\eps_k)t_k^{(0)}\\
	& = (\tfrac{1+\gamma}{2}+\gamma+\eps_k)iR - (\gamma+\eps_k)t_k^{(i)}\\
	& \geq 2\eps_k t_k^{(0)} - (\gamma+\eps_k)t_k^{(i)}\\
	& \geq - (\gamma+\eps_k)t_k^{(i)}
\end{align*}
while $t_k^{(i)} = t_k^{(0)} - R\lceil \frac{4\eps_k t_k^{(0)}}{R}\rceil > \frac{1}{3}t_k^{(0)}$, contradicting our choice of $t_k^{(0)}$.

By construction, \eqref{minusr2} holds for $t_k=t_k^{(i)}$ when the procedure stops, and since $t_k=t_k^{(i)}\geq\tfrac{t_k^{(0)}}{3}$, the right-hand side inequality in \eqref{limsup2} also holds.

Moreover, by induction, for all $i\geq 0$,
\[
\frac{r_\psi(t_k^{(i)})}{t_k^{(i)}} \geq -\gamma - \eps_k.
\]
Indeed,
\begin{align*}
r_\psi(t_k^{(i)}) & \geq r_\psi(t_k^{(i-1)}) + \frac{1+\gamma}{2}R\\
& \geq -(\gamma+\eps_k)(t_k^{(i)}+R) + \frac{1+\gamma}{2}R\\
& > -(\gamma+\eps_k) t_k^{(i)}.
\end{align*}
So \eqref{limsup2} holds.

And by induction again,
\[
\forall t\geq t_k^{(i)},\quad r_\psi(t) \leq r_\psi(t_k^{(i)}) + R.
\]
Indeed, for $t\geq t_k^{(i-1)}$, one has $r_\psi(t) \leq r_\psi(t_k^{(i-1)}) + R \leq r_\psi(t_k^{(i)})+R$, while for $t$ in $[t_k^{(i)},t_k^{(i-1)}]$, since $t\mapsto t+r_\psi(t)$ is increasing, we may bound
\[
r_\psi(t) \leq r_\psi(t_k^{(i-1)})+t_k^{(i-1)}-t \leq r_\psi(t_k^{(i-1)}) + R \leq r_\psi(t_k^{(i)}) + R.
\]
This proves \eqref{aftertk2}.
\end{proof}

Now we may proceed with the definition of our Cantor set.
Set $E_0=[0,1)^n$, and assume that $E_{l-1}$ has been defined so that the above properties~(A$_{l-1}$) and (B$_{l-1}$) hold.
We fix a cube $C$ in $E_{l-1}$, divide it into $N^{n}$ subcubes of sidelength $N^{-l}$ and explain which among those subcubes will belong to $E_{l}$.
Denote by $E_{l}(C)$ the collection of these subcubes.

Letting $R_0=\max(4n,n^2)$, $R_1=\frac{10R_0}{1-\gamma}$ and $R_2 = 2n(R_1+6R_0) + 1$, we define
\[
l_k^- = \lceil\frac{t_k^--4R_0M_k}{M}\rceil,
\quad\mbox{and}\quad
l_k^+ = \lfloor\frac{t_k+R_2M_k}{M}\rfloor.
\]
We shall distinguish two cases:

\begin{description}
    \item[Case 1] $l_{k-1}^{+}< l \leq l_{k}^-$\\
Set $E_{l}(C)$ to be the set of subcubes $C'\subset C$ such that for all $x$ in $C'$,
\[
c_x(lM) \geq -M_k + M.
\]
    \item[Case 2] $l_{k}^- < l\leq l_{k}^+$\\
Let $x_k$ denote the center of the unique cube $C_{l_k^-}$ of level $l_k^-$ containing~$C$, and note that $\lambda_{1}(a_{l_k^-M}u_{x_k}\Z^{n+1})\geq e^{-M_{k}}$.
By Lemma~\ref{ratnearbad} below applied at time $t=l_k^-M$ and with parameter $R=2M_k$, there exists a rational point $v_k\in C_{l_{k}^-}$ such that 
\begin{equation*} 
    e^{l_{k}^-M-2M_{k}} \leq H(v_k) \leq e^{l_{k}^-M+4nM_{k}}.
\end{equation*}
With our choice of $R_0$ and the definition of $l_k^-$, this implies
\[
e^{t_k^--5R_0M_k} \leq H(v_k) \leq e^{t_k^--3R_0M_k}.
\]
Pick $y_{k}\in C_{l_{k}^-}$ such that
\begin{equation*}
    d(v_k,y_{k})=(1-\frac{1}{2k}) \psi(H(v_k)).
\end{equation*}
For each \(l \in \{l_k^-,\dots,l_k^+\}\), take
\[
E_{l}(C)=\{C_{l}(y_k)\}
\]
where $C_l(y_k)$ denotes the unique cube of level $l$ containing $y_k$.
\end{description}

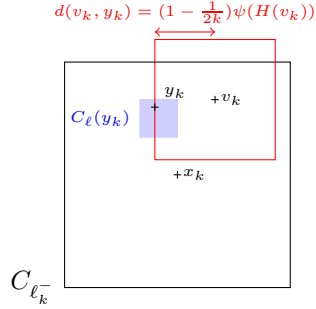
\begin{figure}[h!]
\label{yzero}
\begin{center}
\begin{tikzpicture}
\draw (0,0) -- (0,3) -- (3,3) -- (3,0) -- (0,0) node[left] {$C_{l_k^-}$};
\draw (1.45,1.5) -- (1.55,1.5);
\draw (1.5,1.45) -- (1.5,1.55);
\draw (1.5,1.5) node[right] {\tiny{$\!x_k$}};

\draw (1.95,2.5) -- (2.05,2.5);
\draw (2,2.45) -- (2,2.55);
\draw (2,2.5) node[right] {\tiny{$\!v_k$}};

\filldraw[color=blue!75!white!25!] (1,2) -- (1.5,2) -- (1.5,2.5) -- (1,2.5) -- (1,2);
\draw[color=blue] (1,2.25) node[left] {\tiny{$C_l(y_k)$}};

\draw[color=red] (1.2,1.7) -- (1.2,3.3) -- (2.8,3.3) -- (2.8,1.7) -- (1.2,1.7);
\draw[color=red, <->] (1.2,3.4) -- (1.6,3.4) node[above] {\tiny{$d(v_k,y_k)=(1-\frac{1}{2k})\psi(H(v_k))$}} -- (2,3.4);

\draw (1.15,2.4) -- (1.25,2.4);
\draw (1.2,2.35) -- (1.2,2.45);
\draw (1.2,2.4) node[above right] {\tiny{$y_k$}};

\end{tikzpicture}
\end{center}
\caption{Choice of $E_l(C)$ for $l\in\{l_k^-,\dots,l_k^+\}$}
\end{figure}

Let us check by induction that if $E_l(C)$ is chosen as explained above, then properties \ref{al} and \ref{bl} hold for $l$ arbitrarily large.

\begin{description}
\item[Case 1]
Condition \ref{al} is satisfied by our choice of $E_l(C)$, and \ref{bl} is satisfied because it coincides with (B$_{l-1}$).
\item[Case 2]
Assuming that (A$_{l_k^-}$) and (B$_{l_k^-}$) hold for $C$ in $E_{l_k^-}$, let us show that (A$_{l_k^+}$) and (B$_{l_k^+}$) hold for every $x$ in $E_{l_k^+}(C)=C_{l_k^+}(y_k)$.
Let us start with ($B_{l_k^+}$).
By construction,
\[
e^{t_k^--5R_0M_k} < H(v_k) < e^{t_k^--3R_0M_k}.
\]
Let $\bv_k$ be a vector in $\Z^{n+1}$ corresponding to the rational point $v_k$ in $C_{l_k^-}$.
For $x$ in $C_{l_k^+}(y_k)$, one has
\begin{align*}
d(x,v_k) & \leq d(x,y_k) + d(y_k,v_k)\\
& \leq e^{-\frac{n+1}{n}l_k^+M} + (1-\frac{1}{2k})\psi(H(v_k)).
\end{align*}
Recalling that $H(v_k)\leq e^{t_k^-}=\Psi^{-1}(e^{t_k})$, one finds $e^{-\frac{n+1}{n}t_k}\leq\psi(H(v_k))$, and since $l_k^+ M\geq t_k+M_k$ this yields
\[
d(x,v_k) \leq e^{-M_k}\psi(H(v_k)) + (1-\frac{1}{2k})\psi(H(v_k)) < \psi(H(v_k)),
\] 
provided we chose $t_k$ large enough to ensure $e^{-M_k}<\frac{1}{2k}$.
Similarly,
\[
d(x,v_k) \geq (1-\frac{1}{2k})\psi(H(v_k)) - e^{-M_k}\psi(H(v_k)) \geq (1-\frac{1}{k})\psi(H(v_k)).
\]
We have checked that conditions (i) and (ii) in (B$_{l_k^+}$) hold, and by Lemma~\ref{c1} and our choice of the sequence $(t_k)$, condition (iii) follows automatically.

It remains to show that $\bv_k$ achieves the first minimum of $a_tu_x\Z^{n+1}$ for $t$ in $[t_k^-,t_k^+]$.
At time $T=t_k^--4R_0M_k$, we have
\(
\lambda_1(a_Tu_x\Z^{n+1}) \geq e^{-M_k}
\)
and therefore, by Minkowski's second theorem,
\[
\lambda_{n+1}(a_Tu_x\Z^{n+1}) \lesssim e^{nM_k}.
\]
Since the largest eigenvalue of $a_s$ is equal to $e^{\frac{s}{n}}$, we infer, for all $s\geq 0$,
\[
\lambda_{n+1}(a_{T+s}u_x\Z^{n+1}) \lesssim e^{nM_k+\frac{s}{n}}.
\]
On the other hand, for $T+s\leq t_k^x$,
\[
\lambda_1(a_{T+s}u_x\Z^{n+1})
	\leq \norm{a_{T+s}u_x\bv_k}
	= e^{-s}\norm{a_Tu_x\bv_k}
	\leq e^{-s + R_0M_k}.
\]
Since one always has $\lambda_2 \gtrsim \lambda_1^{-1}\lambda_{n+1}^{-n+1}$, this yields
\[
\lambda_2(a_{T+s}u_x\Z^{n+1})
	\gtrsim e^{s-R_0M_k} e^{-n(n-1)M_k-\frac{(n-1)s}{n}}
	\geq e^{\frac{s}{n} - 2R_0M_k},
\]
and if $s > 3R_0M_k$, we find
\[
\lambda_2(a_{T+s}u_x\Z^{n+1}) \geq  e^{-2R_0M_k}
	> e^{-s + R_0M_k}
	\geq \norm{a_{T+s}u_x\bv_k}.
\]
This implies that for $t$ in $[t_k^-,t_k^x]$, the vector $\bv_k$ achieves the first minimum of $a_tu_x\Z^{n+1}$.
Note also that at time $t_k^x=T+s_k^x$,
\[
\frac{\lambda_2(a_{t_k^x}u_x\Z^{n+1})}{\lambda_1(a_{t_k^x}u_x\Z^{n+1})}
\geq
e^{\frac{n+1}{n}s_k^x-3R_0M_k}
> e^{\frac{n+1}{n}(R_1+R_2)M_k}
\]
so $\bv_k$ continues to achieve $\lambda_1(a_tu_x\Z^{n+1})$ on the interval $[t_k^x,t_k^x+(R_1+R_2)M_k]$, which contains $[t_k^x,t_k^+]$.
Indeed, for $s\in [0,(R_1+R_2)M_k]$,
\begin{align*}
    \lambda_{2}(a_{t_k^x+s}u_x\Z^{n+1})&\geq e^{-s}\lambda_2(a_{t_k^x}u_x\Z^{n+1})\\
    &\geq e^{-s}e^{\frac{n+1}{n}(R_1+R_2)M_k}\lambda_1(a_{t_k^x}u_x\Z^{n+1})\\
    &\geq e^{\frac{s}{n}}\norm{a_{t_{k}^x}u_x \bv_k}\\
    &\geq \norm{a_{t_{k}^x+s}u_x \bv_k}.
\end{align*}
Finally, to prove that (A$_{l_k^+}$) holds, we only need to check that for all $t$ in $[t_k^+,l_k^+M]$, $c_x(t) > M_{k+1}$.
Write, for $t$ in $[t_k^+,l_k^+M]$,
\begin{align*}
c_x(t) & = c_x(t_k^x) + \frac{t-t_k^x}{n}\\
	& \geq r_\psi(t_k) - 5R_0M_k + \frac{t-t_k}{n}\\
	& \geq M_{k+1} - R_1M_k - 5 R_0M_k + \frac{R_2M_k}{2n}\\
	& > M_{k+1} + M
\end{align*}
since we chose $R_2 > 2n(R_1+6R_0)$.
\end{description}

We conclude this paragraph with the lemma used in Case~2 above, to obtain a rational point of controlled height inside any cube $C$ from $E_{l_k^-}$.
\begin{lemma}[Rational points near badly approximable points]
\label{ratnearbad}
Given parameters $t>0$ and $R\geq 1$, assume $x\in[0,1)^{n}$ is such that $\lambda_1(a_tu_x \Z^{n+1}) \geq e^{-R}$.
Then there exists a rational point $v=v(x)$ such that:
\begin{enumerate}
\item $e^{t-R} \leq H(v) \leq e^{t+2nR}$;
\item $d(x,v) \leq \frac{1}{2}e^{-\frac{n+1}{n}t}$.
\end{enumerate}
\end{lemma}
\begin{proof}
By Minkowski's first theorem applied to the lattice $a_{t+2nR}u_x\Z^{n+1}$, there exists  $\bv$ in $\Z^{n+1}$ such that
\[
\norm{a_{t+2nR}u_x\bv} \leq 1.
\]
The point $v$ in $\R^n$ corresponding to $\bv$ satisfies
\begin{equation}\label{tplusnr}
H(v)\leq e^{t+2nR}
\quad\mbox{and}\quad
d(x,v) \leq e^{-\frac{t}{n}-2R}H(v)^{-1}.
\end{equation}
On the other hand, our assumption that $\lambda_1(a_tu_x\Z^{n+1})\geq e^{-R}$ implies that
\[
\norm{a_tu_x\bv} = H(v) \max(e^{-t},e^{\frac{t}{n}} d(x,v)) \geq e^{-R} 
\]
and since $d(x,v) \leq e^{-\frac{t}{n}-2R}H(v)^{-1}$, we must have
\[
H(v)\geq e^{t-R}.
\]
Going back to \eqref{tplusnr} this yields
\[
d(x,v) \leq e^{-\frac{t}{n}-2R} e^{-t+R} < \frac{1}{2}e^{-\frac{n+1}{n}t}.
\]
\end{proof}

\subsection{Branching and Hausdorff dimension}

It now remains to compute the Hausdorff dimension of $E_\infty$.
For that, we first need to obtain a good lower bound on the branching of the Cantor set at each level $l$ between $t_{k-1}^+$ and $t_k^-$.
This will be a consequence of Lemma~\ref{branching} below, whose proof uses a variant of the well-known Simplex Lemma originating in the works of Davenport and Schmidt~\cite[page 57]{schmidt_da}.

Recall that the Cantor set $E_\infty$ is obtained as a decreasing intersection $E_\infty=\bigcap_{l\geq 1} E_l$, where each $E_l$ is a finite union of disjoint cubes of side length $N^{-l}$, where $N$ is a large integer given as
\[
N=e^{\frac{(n+1)M}{n}}.
\]

\begin{lemma}[Large branching for $l_{k-1}^+<l\leq l_k^-$]
\label{branching}
There exists a constant $R_3$ depending on $n$ such that for all large enough $k$, if $l_{k-1}^+<l\leq l_k^-$ and $C$ is any cube in $E_{l-1}$, then
\[
\card E_l(C) \geq N^n-R_3 N^{n-\frac{1}{n+1}}.
\]
\end{lemma}
\begin{proof}
Set $t=(l-1)M$.
Since $l_{k-1}^+<l\leq l_k^-$ and $C$ belongs to $E_{l-1}$, we have
\[
\forall x\in C,\quad \lambda_1(a_tu_x\Z^{n+1}) > e^{-M_k+M}.
\]
Let $x_0$ be the bottom left corner of $C$, and define
\begin{equation*}
    S_{C}:=\{\bv\in \Z^{n+1}: \|a_{t}u_{x_0}\bv\|<e^{-M_{k}+3M}\}.
\end{equation*}
We claim that there exists some hyperplane $H_C$ in $\R^{n+1}$ such that $S_{C}\subseteq H_C$.
(This statement can be viewed as a version of the Simplex Lemma.)
Otherwise, there would exist linearly independent vectors $\bv_{1},\dots, \bv_{n+1}\in \Z^{n+1}$ such that 
\begin{equation*}
    \|a_{t}u_{x_{0}}\bv_{i}\|<e^{-M_{k}+3M} \text{ for } 1\leq i\leq n+1.
    \end{equation*}
But $a_{t}u_{x_{0}}\Z^{n+1}$ is a lattice of covolume $1$, hence it cannot have $n+1$ linearly independent vectors of norm less than $1$.
This yields the desired contradiction and proves our claim as soon as $M_k>3M$.

Now let
\[
A_C = \{x\in C\ ;\ d([e_1],a_tu_x H_C) \leq e^{-\frac{M}{n}}\}.
\]
Let $x$ be a point in $C\setminus A_C$ and $\bv$ any non-zero vector in $\Z^{n+1}$.
If $\bv\in S_C$, then $\bv\in H_C$ so the projection of $a_tu_x\bv$ to $e_1^{\perp}$ has norm at least $e^{-\frac{M}{n}}\norm{a_tu_x\bv}$, and since this coordinate is expanded by a factor $e^{\frac{s}{n}}$ under the action of $a_s$, we have
\begin{align*}
\norm{a_{t+M}u_x\bv} & \geq e^{\frac{M}{n}}e^{-\frac{M}{n}}\norm{a_tu_{x}\bv}\\
	& \geq e^{-M_k+M}.
\end{align*}
On the other hand, if $\bv\not\in S_C$, we can bound
\[
\norm{a_{t+M}u_x\bv} \geq e^{-M}\norm{a_tu_x\bv}
\geq e^{-2M}\norm{a_tu_{x_0}\bv}
\geq e^{-M_k+M}.
\]
This shows that
\[
\forall x\in C\setminus A_C,\quad c_x(lM) \geq -M_k+M.
\]
To conclude the proof, we show that $A_C$ is included in a neighborhood of size $N^{-l} e^{-\frac{M}{n}}$ of a hyperplane in $\R^n$.
For $x$ in $C$, we let $y=e^{\frac{n+1}{n}lM}(x-x_0)$ in $[0,1)^n$ so that
\[
a_tu_x = u_ya_tu_{x_0}.
\]
Let $\phi_{t,C}$ be a linear form of norm $1$ vanishing on $H_{t,C}=a_tu_{x_0}H_C$; then
\[
d([e_1],[a_tu_x H_C]) = d([e_1], [u_y H_{t,C}])
\asymp d([u_{-y}e_1],H_{t,C}) 
\asymp \phi_{t,C}(u_{-y}e_1).
\]
Therefore $x\in A_C$ implies $\phi_{t,C}(u_{-y}e_1)\lesssim e^{-\frac{M}{n}}$, and this inequality means that $y$ lies in a neighborhood of size $O(e^{-\frac{M}{n}})$ of the affine hyperplane of $\R^n$ defined by the equation $\phi_{t,C}(e_1-y_1e_2-\dots-y_ne_{n+1})=0$.
Equivalently, $x$ lies in a neighborhood of size $O(N^{-l}e^{-\frac{M}{n}})$ of an affine hyperplane in $\R^n$ .
The number of subcubes of $C$ that meet this neighborhood is bounded above by
\[
\lesssim_n N^{n} e^{-\frac{M}{n}}
\]
so
\[
\card E_l(C) \geq N^n(1-O_n(e^{-\frac{M}{n}})) = N^n - O_n(N^{n-\frac{1}{n+1}}).
\]
\end{proof}

The lower bound on the branching given by the above lemma is all that is needed to get a good lower bound on the Hausdorff dimension of $E_\infty$, and therefore on $E(\psi)$.

\begin{proof}[Proof of Theorem~\ref{thm:main}]
To get a lower bound on the Hausdorff dimension of $E_\infty$, we use the Mass distribution principle~\cite[\S4.2]{falconer}, but first we replace $E_\infty$ by a more regular Cantor subset, to simplify later computations.

For $l\geq 1$, define
\[
b_l = 
\left\{
\begin{array}{ll}
\lfloor N^n(1-R_3e^{-\frac{M}{n}})\rfloor & \mbox{if}\ l_{k-1}^+ < l \leq l_k^-\\
1 & \mbox{if}\ l_k^- < l \leq l_k^+.
\end{array}
\right.
\]
Removing some cubes in $E_l$ at each step, one obtains a Cantor subset $F_\infty\subset E_\infty$ given as
\[
F_\infty = \bigcap_{l\geq 1}F_l,
\]
where each cube $C$ in $F_{l-1}$ contains exactly $b_l$ subcubes in $F_l$.
By \cite[Proposition~1.7]{falconer}, there exists a probability measure $\mu$ supported on $F_\infty$ such that for any level $l$ cube $C\subset F_l$,
\[
\mu(C) = (b_1\dots b_l)^{-1}.
\]
We claim that for $\alpha < \liminf_{l\to\infty} \frac{\log(b_1\dots b_l)}{l\log N}$, there exists $C=C_{n,N,\alpha}$ such that
\[
\forall x\in\R,\ \forall r>0,\quad
\mu(B(x,r)) \leq C r^\alpha.
\]
Indeed, picking $l$ so that $N^{-l} < r \leq N^{-l+1}$, the ball $B(x,r)$ meets at most $(3N)^n$ cubes from $F_l$, and therefore
\[
\mu(B(x,r)) \leq (3N)^n (b_1\dots b_l)^{-1} \leq (3N)^n N^{-l \alpha} \leq (3N)^n r^{\alpha}
\]
provided $l$ is large enough, or equivalently $r$ small enough in terms of $n$, $N$ and $\alpha$.
By the Mass distribution principle, this implies
\[
\dim_H F_\infty \geq \alpha.
\]
To conclude, observe that 
\[
\liminf_{l\to\infty} \frac{\log(b_1\dots b_l)}{l\log N} = \lim_{k\to\infty}\frac{\log(b_1\dots b_{l_k^+})}{l_k^+\log N}
= \lim_{k\to\infty}\frac{\log(b_1\dots b_{l_k^-})}{l_k^+\log N}
\]
and if the sequences $(l_k^{\pm})_{k\geq 1}$ satisfies $l_{k-1}^+ = o(l_k^-)$ --- this can always be ensured by taking a sequence $(t_k)$ increasing sufficiently fast --- this limit is bounded below by
\begin{align*}
\lim_{k\to\infty} \frac{l_k^- \log\lfloor N^n-R_3N^{n-\frac{1}{n+1}}\rfloor}{l_k^+\log N} 
	& = \frac{\log\lfloor N^n-R_3N^{n-\frac{1}{n+1}}\rfloor}{\log N} \lim_{k\to\infty} \frac{t_k^-}{t_k}\\
	& = \frac{\log\lfloor N^n-R_3N^{n-\frac{1}{n+1}}\rfloor}{\log N} \frac{n+1}{n\lambda_\psi}.
\end{align*}
\details{Recall $t_k^- = t_k+r_\psi(t_k)$ and $r_\psi(t_k)\sim -\gamma_\psi t_k = -\frac{n\lambda_\psi-n-1}{n\lambda_\psi} t_k$, so $t_k^-\sim \frac{(n+1)t_k}{n\lambda_\psi}$.}
Thus
\[
\dim_H E(\psi) \geq \dim_H F_\infty \geq \frac{\log\lfloor N^n-R_3N^{n-\frac{1}{n+1}}\rfloor}{\log N} \frac{n+1}{n\lambda_\psi}
\]
and as $N$ (or equivalently $M$) tends to $+\infty$, this yields the desired result
\[
\dim_H E(\psi) \geq \frac{n+1}{\lambda_\psi}.
\]
\end{proof}

Theorem~\ref{thm:existence} is now an easy consequence of Theorem~\ref{thm:main}, and of Jarn{\'i}k's results.
We even get the following slightly more precise result.

\begin{theorem}[Existence of exact $\psi$-approximable vectors]
\label{thm:existence general}
    Let $n\geq 1$ be an integer.
    If $\psi\colon\R^+\to\R^+$ is non-increasing and satisfies $\lim_{s\to\infty}s^{\frac{n+1}{n}}\psi(s)=0$, then the set $E(\psi)$ is uncountable.
\end{theorem}
\begin{proof}
If $\lambda_\psi<+\infty$, then Theorem~\ref{thm:main} shows that $\dim_H E(\psi) = \frac{n+1}{\lambda_\psi}> 0$, so the result is clear.
If $\lambda_\psi=+\infty$, then $\lim_{s\to\infty} s^2\psi(s)=0$ and so we may use Jarn{\'i}k's construction to get that $E(\psi)$ is uncountable, see also~\cite[Theorem~J]{bugeaud3}.
\end{proof}

\section*{Conclusion}

The method developed here to study exact approximation of points in $\R^n$ is robust, and can be adapted to study a number of similar problems.

\bigskip
\noindent\textbf{Other norms on $\R^n$.}
In the definition of the set $W(\psi)$ of $\psi$-approximable points in $\R^n$, we used the norm on $\R^n$ given by $\norm{x}=\max_{1\leq i\leq n}\abs{x_i}$, because this is the standard setting for simultaneous diophantine approximation, and the one studied by Jarn{\'i}k in his foundational paper~\cite{jarnik}.
But one could also use the Euclidean norm on $\R^n$, or any other norm $N$, and study the corresponding notion of exact $\psi$-approximability.
It is not difficult to check that Theorems~\ref{thm:existence} and \ref{thm:main} are still valid in this slightly more general context.
The main difference lies in the fact that for the Dani correspondence, one should endow the space $\R^{n+1}$ with the norm 
\[
\norm{x'} = \max(\abs{x_0},N(x)),
\]
if $x'=(x_0,x)$ in the identification $\R^{n+1}\simeq\R\times\R^n$.

\bigskip
\noindent\textbf{Approximation of linear forms and matrices.}
Given a non-increasing function $\psi\colon\R^+\to\R^+$ and positive integers $m$ and $n$, one defines 
\[
W(m,n;\psi)=\left\{X\in M_{m\times n}(\R):
\begin{array}{ll}
     &\|\bq X-\bp\|<\psi(\|\bq\|)\|\bq\|\\
     & \mbox{for infinitely many} (\bp,\bq)\in \Z^{n}\times \Z^{m}.
\end{array}
\right\}
\]
The matrices in $W(m,n;\psi)$ are usually called $\psi$-approximable $m\times n$ matrices. 
For a decreasing function $\psi$, Dodson \cite{dodson} showed that $W(m,n;\psi)$ has Hausdorff dimension $(m-1)n+\frac{m+n}{\lambda_{\psi}}$ provided $\lambda_\psi\geq\tfrac{m+n}{n}$.
This general setting is in fact the one used by Beresnevich, Dickinson and Velani~\cite{bdv} to study exact approximability.
The techniques of the present paper can be used to show that if $s^{\frac{m+n}{n}}\psi(s)$ tends to zero as $s$ goes to infinity, then the set $E(m,n;\psi)$ of $m\times n$ matrices that are exactly $\psi$-approximable has the same Hausdorff dimension as $W(m,n;\psi)$.

\bigskip
\noindent\textbf{Intrinsic diophantine approximation on manifolds.}
Let $X$ be an algebraic variety defined over $\Q$ and such that $X(\Q)$ is dense in $X(\R)$.
Having fixed a distance on $X(\R)$ and a height on $X(\Q)$, one may study the quality of approximation by points in $X(\Q)$ to points in $X(\R)$.
This problem is usually referred to as \emph{intrinsic diophantine approximation} on $X$. 

In a number of cases, it has been shown that Jarn{\'i}k's theorem holds in this context; this is so for instance when $X$ is a quadric hypersurface \cite{km_sphere,fkms} or when $X$ is a Grassmann variety \cite{saxce_grassmannienne} or any flag variety \cite{saxce_hdr}.
In all those cases, there exists some version of the Dani correspondence that interprets diophantine properties in terms of diagonal orbits in spaces of lattices, so it is natural to expect that the analog of Theorem~\ref{thm:main} holds, and can be obtained by methods similar to the ones used here.

\printbibliography
\end{document}